\documentclass[conference]{IEEEtran}
\IEEEoverridecommandlockouts
\usepackage{cite}
\usepackage{amsmath,amssymb,amsfonts,url,bm,enumerate}
\usepackage{amsthm}
\usepackage{graphicx}
\usepackage{textcomp}
\usepackage[dvipsnames]{xcolor}
\usepackage{tikz}
\usetikzlibrary{arrows.meta}
\usepackage{float}
\usepackage{colortbl}
\usepackage{booktabs}
\usepackage{array}
\usepackage{comment}
\usepackage{multicol}
\usepackage{enumitem}
\def\BibTeX{{\rm B\kern-.05em{\sc i\kern-.025em b}\kern-.08em
    T\kern-.1667em\lower.7ex\hbox{E}\kern-.125emX}}
    
\usepackage{hyperref}
\usepackage{pict2e}

\usepackage{tikz-cd}
\tikzset{commutative diagrams/.cd}
\tikzcdset{row sep/tiny=0.12 em}

\usetikzlibrary{intersections}

\newlength\myheight

        \usetikzlibrary{calc}
\newcommand*\circled[1]{\tikz[baseline=(char.base)]{
        \node[shape=circle,draw,minimum size=4mm, inner sep=0pt] (char)
        {\rule[-3pt]{0pt}{\dimexpr2ex+2pt}#1};}}

\theoremstyle{plain}
\newtheorem{theorem}{Theorem}
\newtheorem{lemma}[theorem]{Lemma}
\newtheorem{proposition}[theorem]{Proposition}

\newtheorem{corollary}[theorem]{Corollary}

\theoremstyle{definition}
\newtheorem{definition}[theorem]{Definition}

\newtheorem{remark}[theorem]{Remark}

\newtheorem{example}[theorem]{Example}

\newtheorem{claim}{Claim}[theorem]
\def\endclaim{\ensuremath{\blacksquare}} 

\newcommand{\C}{{\mbox{${\mathcal C}$}}}
\DeclareMathOperator{\kernel}{ker}
\DeclareMathOperator{\image}{im}
\newcommand{\set}[2]{\{#1\;|\;#2\}}
\newcommand{\setin}[3]{\{#1\in#2\;|\;#3\}}
\newcommand{\ex}[2]{\exists_{#1}.\,#2}

\newcommand{\KSub}{\ensuremath{\mathrm{KSub}}}

\newcommand{\auv}[1]{``#1''}
\newcommand{\sai}[1]{[\,#1\,]}
\newcommand{\nul}{\ensuremath{\underline{0}}}
\newcommand{\downset}{\mathop{\downarrow}\!}

\newcommand{\komma}{,\hspace{0.3em}}
\newcommand{\adj}{^{*}}
\newcommand{\op}{\ensuremath{^{\mathrm{op}}}}
\newcommand{\idmap}[1][]{\ensuremath{\mathrm{id}_{#1}}}
\newcommand{\conjun}{\mathrel{\wedge}}
\newcommand{\disjun}{\mathrel{\vee}}

\newcommand{\cat}[1]{\ensuremath{\mathbf{#1}}}
\newcommand{\Cat}[1]{\ensuremath{\mathbf{#1}}}

\definecolor{darkgreen}{rgb}{0.0, 0.5, 0.0}

\newcommand{\lin}[1]{[#1]}

\definecolor{brightgreen}{rgb}{0.4, 1.0, 0.0}
\newcolumntype{?}{!{\vrule width 1pt}}

\begin{document}

\title{Foulis m-semilattices and their modules\\
}
\author{
\IEEEauthorblockN{Michal~Botur}
\IEEEauthorblockA{\textit{Department of Algebra and Geometry} \\
	\textit{Faculty of Science, Palack\'y University Olomouc }\\
	Olomouc, Czech Republic\\
	michal.botur@upol.cz}
\and
\IEEEauthorblockN{Jan~Paseka, Milan Lekár}
\IEEEauthorblockA{\textit{Department of Mathematics and Statistics} \\
\textit{Faculty of Science, Masaryk University}\\
Brno, Czech Republic\\
paseka@math.muni.cz}
}

\maketitle

\begin{abstract}
Building upon  results of Jacobs, we show that 
the category $\Cat{OMLatLin}$  of orthomodular lattices and linear maps forms a dagger category. For each orthomodular lattice $X$, we construct a Foulis m-semilattice $\Cat{Lin}(X)$ composed of endomorphisms of $X$. This m-semilattice acts as a quantale, enabling us to regard $X$ as a left $\Cat{Lin}(X)$-module. Our novel approach introduces a fuzzy-theoretic dimension to the theory of orthomodular lattices.
\end{abstract}


\section{Introduction}

Quantale modules offer a robust algebraic framework for studying fuzzy set theory and its generalizations. Originating from Mulvey's pioneering work~\cite{Mulvey86}, quantales emerged as non-commutative generalizations of locales, providing a rich mathematical structure that naturally accommodates various forms of fuzzy logic and reasoning.

Höhle~\cite{Hohle95} first systematically explored the connection between quantale modules and fuzzy sets~\cite{Zad65}, demonstrating that categories of fuzzy sets can be viewed as categories of modules over specific quantales. 
By representing fuzzy sets as elements of a quantale module, we can leverage the algebraic structure to define and analyze fuzzy operations rigorously. 
This perspective has proven instrumental in developing the theoretical foundations of many-valued and fuzzy logics.

Unital quantales, in particular, play a dual role as both truth-value sets for fuzzy logics and coefficient structures for their corresponding categories of modules~\cite{KrPaseka08}. This duality enables an elegant treatment of both the logical and set-theoretic aspects of fuzzy mathematics within a unified framework~\cite{BelohlAvek02}.

The module-theoretic approach offers several advantages: it naturally accommodates various t-norms and their associated logics, provides clear conceptual tools for handling grades of membership, and offers powerful categorical methods for studying fuzzy algebraic structures~\cite{Sol09}. Moreover, quantale modules have extended classical results from fuzzy set theory to more general settings.

Key applications include:

\begin{itemize}
    \item {\bfseries Fuzzy Logic}: Modeling fuzzy logical connectives and inference rules.
    \item {\bfseries Fuzzy Set Operations}: Defining and analyzing operations like union, intersection, and complement.
    \item {\bfseries Fuzzy Topology}: Constructing fuzzy topological spaces.
    \item {\bfseries Image Processing}: Applying fuzzy logic techniques for image processing tasks.
\end{itemize}

In this research, motivated by the aforementioned considerations, 
we investigate Foulis m-semilattices and modules over them, which will serve as quantale-like structures for orthomodular lattices. By doing so, we introduce a natural fuzzy-theoretic aspect to the theory of orthomodular lattices since every quantale 
module obtains a natural fuzzy order (see \cite{Sol09}).

The paper is structured as follows. 
After this introduction, Section~\ref{Preliminaries} 
provides some basic notions, notations, and results that will be used in the article. In particular, we introduce here a category \Cat{OMLatLin} of orthomodular
lattices and linear maps and study its properties. In Section~\ref{Foulis}, we introduce the notion of a Foulis 
m-semilattice and show that the m-semilattice $\Cat{Lin}(X)$ of endomorphisms of an orthomodular lattice $X$ is a Foulis 
m-semilattice. Conversely, every Foulis 
m-semilattice $S$ yields an orthomodular lattice $[S]$ of Sasaki projections of $S$. 

Moreover, in Section \ref{semmod} on m-semilattice modules we show first that, for an orthomodular lattice $X$, $X$ is a left $\Cat{Lin}(X)$-module. Second, the orthomodular lattice $[S]$ of Sasaki projections of a Foulis 
m-semilattice $S$ is a left $S$-module.

Our conclusions follow in Section~\ref{Conclusions}.

In what follows, we assume familiarity with the fundamental concepts and results of lattice and poset theory. For further details on these topics, readers are referred to 
the monographs \cite{Bi} by G.~Birkhoff and \cite{OML} by G.~Kalmbach. For concepts and results on quantales and quantale modules, we direct the reader
to \cite{KrPaseka08}, \cite{rosenthal1}, and \cite{russo}. 
The necessary categorical background can be found in \cite{Joy of cats}; we recommend \cite{Jac} for dagger categories.

\section{Basic concepts}\label{Preliminaries}

\begin{definition}
\label{OMLatDef}
A meet semi-lattice $(X,\conjun 1)$ is called an {\em ortholattice} if it
comes equipped with a function $(-)^{\perp}\colon X \to X$ satisfying:
\begin{itemize}
   \item $x^{\perp\perp} = x$;
   \item $x \leq y$ implies $y^\perp \leq x^\perp$;
   \item $x \conjun x^\perp = 1^\perp$.
\end{itemize}

\noindent One can then define a bottom element as $0 = 1 \conjun
1^{\perp} = 1^\perp$ and join by $x\disjun y = (x^{\perp}\conjun
y^{\perp})^{\perp}$, satisfying $x\disjun x^{\perp} = 1$.

We write $x\perp y$ if and only if $x\leq y^{\perp}$. 

Such an ortholattice is called {\em orthomodular} if it satisfies (one of)
the three equivalent conditions:
\begin{itemize}
\item $x \leq y$ implies $y = x \disjun (x^\perp \conjun y)$;

\item $x \leq y$ implies $x = y \conjun (y^\perp \disjun x)$;

\item $x \leq y$ and $x^{\perp} \conjun y = 0$ implies $x=y$.
\end{itemize}
\end{definition}

\begin{definition}
\label{DagcatDef}
A {\it dagger} on a category \C\ is a functor ${}^\star \colon \C\op \to \C$ that is involutive and the identity on objects. A category equipped with a dagger is called a {\it dagger category}.

Let \C\ be a dagger category. A morphism $f \colon A \to B$ is called a {\it dagger monomorphism} if $f^{\star} \circ f={\idmap}_A$, and $f$ is called a {\it dagger isomorphism} if 
$f^{\star} \circ f = {\idmap}_A$ and $f \circ f^\star = {\idmap}_B$. 
\end{definition}

We now introduce a new way of organising orthomodular lattices
into a dagger category.

\begin{definition}
The category \Cat{OMLatLin} has orthomodular
lattices as objects.
A morphism $f \colon X\rightarrow Y$ in \Cat{OMLatLin} is a 
function $f \colon X\rightarrow Y$ between the underlying sets such that 
there is a function $h \colon Y \to X$ and, 
for any $x \in X$ and $y \in Y$,
\[ f(x) \perp y \text{ if and only if } x \perp h(y). \]
In this case we say that $h$ is an {\it adjoint} of a {\em linear map} $f$.

It is clear that adjointness is a symmetric property: if a map $f$ possesses an adjoint $h$, then $f$ is also an adjoint of $h$. 
We denote $\Cat{Lin}(X,Y)$ the set of all linear maps from $X$ to $Y$.
If $X=Y$ we put $\Cat{Lin}(X)=\Cat{Lin}(X,X)$.

Moreover, a map $f \colon X \to X$ is called {\it self-adjoint} if $f$ is an adjoint of itself.

The identity morphism on $X$ is the self-adjoint identity map $\idmap \colon X\rightarrow X$. Composition of $\smash{X \stackrel{f}{\rightarrow} Y
  \stackrel{g}{\rightarrow} Z}$ is given by usual composition of maps.
\end{definition} 

We immediately see that \Cat{OMLatLin} is really a category. Namely, if $h$ is an adjoint 
of $f$ and $k$ is an adjoint of $g$ we have, for any $x \in X$ and $z \in Z$,
\begin{align*}
    g(f(x)) \perp z &\text{ if and only if } f(x) \perp k(z) \\ &\text{ if and only if }  
    x \perp h(k(z)).
\end{align*}
Hence $h\circ k$ is an adjoint of $g\circ f$. Moreover, for any $x, y \in X$, 
\begin{align*}
    \idmap(x) \perp y &\text{ if and only if } x \perp y\\  &\text{ if and only if }  
    x \perp \idmap(y)
\end{align*}
and $\idmap \colon X\rightarrow X$ is self-adjoint. 

Our guiding example is the following.

\begin{example}
Let $H$ be a Hilbert space. We denote the closed subspace spanned by a subset $S \subseteq H$ by $\lin S$. Let $C(H) = \{ \lin S \colon S \subseteq H \}$. Then $C(H)$ is an orthomodular lattice 
such that $\wedge=\cap$ and $P^{\perp}$ is the orthogonal complement of a closed subspace $P$
of $H$.

Let $f \colon H_1 \to H_2$ be a bounded linear map between Hilbert spaces 
and let $f\adj$ be the usual adjoint of $f$. Then the induced map 
$C(H_1) \to C(H_2) \komma \lin S \mapsto \lin{f(S)}$ has the adjoint 
$C(H_2) \to C(H_1) \komma \lin T \mapsto \lin{f\adj(T)}$.
\end{example}

\begin{lemma} \label{lem:lattice-adjoint}
Let $f \colon X \to Y$ and $h \colon Y \to X$ be maps between orthomodular
lattices. Then the following conditions are equivalent:
\begin{itemize}

\item[\rm (i)] $f\colon (X,\leq) \to (Y,\leq)$ has 
$\hat{h}={}^{\perp}\circ h\circ {}^{\perp}\colon (Y,\leq) \to (X,\leq)$
a right (order)-adjoint map.

\item[\rm (ii)] $f$ possesses the adjoint $h \colon Y \to X$. 
\end{itemize}
\end{lemma}
\begin{proof} Ad (i)$\Rightarrow$ (ii): We have, for any $x \in X$ and $y \in Y$, 
\begin{align*}
   f(x) \perp y &\text{ if and only if }  f(x)\leq y^{\perp} \text{ if and only if }   
   x \leq \hat{h}(y^{\perp})\\  
   &\text{ if and only if } x \leq h(y)^{\perp} \text{ if and only if }
    x \perp h(y).
\end{align*}
Ad (ii)$\Rightarrow$ (i): We have, for any $x \in X$ and $y \in Y$, 
\begin{align*}
    f(x)\leq y &\text{ if and only if }  f(x) \perp y^{\perp} \text{ if and only if } 
    x \perp h(y^{\perp})\\  &\text{ if and only if }  
    x \leq h(y^{\perp})^{\perp} \text{ if and only if } x \leq \hat{h}(y).
\end{align*}
\end{proof}

Since order adjoints are unique and ${}^{\perp}$ is involutive, if $f$ does have an adjoint, it has exactly one adjoint and we shall denote it by $f\adj$. 

Moreover, as usual, we define the {\it kernel} $\kernel f$ and the {\it range} $\image f$
of \mbox{$f\in \Cat{Lin}(X,Y)$,} respectively, by
\begin{align*}
\kernel f \;=\; & \{ x \in X \colon f(x) = 0 \}, \\
\image f \;=\; & \{ f(x) \colon x \in X\}.
\end{align*}

\begin{corollary} \label{cor:lattice-adjoint}
Let $f \colon X \to Y$ be a map between orthomodular
lattices and assume that $f$ possesses the adjoint $h \colon Y \to X$. Then 
$f$ preserves arbitrary existing joins in $X$. In particular, 
$f$ preserves finite joins, is order-preserving and $f(0)=0$. 
\end{corollary}

\begin{theorem}\label{OMLisdagger}
    \Cat{OMLatLin} is  a dagger category.
\end{theorem}
\begin{proof}
    Since every morphism in \Cat{OMLatLin} has a unique adjoint we obtain that 
    ${}\adj\colon \Cat{OMLatLin}\op \to \Cat{OMLatLin}$ is functorial, involutive and the identity on objects.
\end{proof}

Let us recall the following elementary results and definition.

\begin{lemma} {\rm\cite[Lemma 3.4]{Jac}}
\label{DownsetLem}
Let $X$ be an orthomodular lattice, with element $a\in X$. 
The (principal) downset $\downset a = \setin{u}{X}{u \leq a}$ is
  again an orthomodular lattice, with order, meets and
  joins as in $X$, but with its own orthocomplement $\perp_a$ given
  by $u^{\perp_a} = a \conjun u^{\perp}$, where $\perp$ is the
  orthocomplement from $X$.
\end{lemma}

\begin{definition}\label{def:Sasaki projection}
		Let $X$ be an orthomodular lattice. Then the map $\pi_a:X\to X$, $y\mapsto a\wedge(a^\perp\vee y)$ is called the \emph{Sasaki projection} to $a\in X$.
	\end{definition}

 We need the following facts about Sasaki projections 
 (see \cite{LiVe}):
	
	\begin{lemma} {\rm\cite[Lemma 2.7]{LiVe}}\label{lem:Sasaki projection facts}
		Let $X$ be an orthomodular lattice, and let $a\in X$. Then for each $y,z\in X$ we have
		\begin{itemize}
			\item[(a)] $y\leq a$ if and only if $\pi_a(y)=y$;
			\item[(b)] $\pi_a(\pi_a(y^\perp)^\perp))\leq y$;
			\item[(c)] $\pi_a(y)=0$ if and only if $y\leq a^\perp$;
			\item[(d)] $\pi_a(y)\perp z$ if and only if $y\perp \pi_a(z)$.
		\end{itemize}
	\end{lemma}

 \begin{corollary}\label{Sasself}
     Let $X$ be an orthomodular lattice with $a\in X$. Then 
     $\pi_a$ is self-adjoint and idempotent.
 \end{corollary}

Similarly as in \cite[Lemma 3.4]{Jac} we have the following lemma. 

\begin{lemma} \label{DownsetLemma}
Let $X$ be an orthomodular lattice, with element $a\in X$. 
There is a dagger monomorphism $\downset a \rightarrowtail X$ in
  \Cat{OMLatLin}, for which we also write $a$, with
$$\begin{array}{r@{}c@{}l@{}c@{}r@{}c@{}l}
a(u) & = & u \ \text{\ for all}\ u\leq a
& \quad\mbox{and}\quad &\ \ 
a^{*}(x) & = & \pi_a(x) \ \text{\ for all}\ x\in X.
\end{array}$$
\end{lemma}
\begin{proof} 
We compute with the help of Lemma \ref{lem:Sasaki projection facts} (a) and (d): 
\begin{align*}
    a(u) \perp x &\text{ if and only if } \pi_a(u) = u \perp x  \\
    &\text{ if and only if }  
    u \perp \pi_a(x)\\
    &\text{ if and only if } u \leq \pi_a(x)^{\perp}\\ %
    &\text{ if and only if } u \leq \pi_a(x)^{\perp} \wedge a= \pi_a(x)^{\perp_a}\\
    &\text{ if and only if } u \leq \pi_a(x)^{\perp} \wedge a= \pi_a(x)^{\perp_a}.
\end{align*}
The first equivalence follows from Lemma \ref{lem:Sasaki projection facts}(a), 
the second one from Lemma \ref{lem:Sasaki projection facts}(d). The remaining equivalences are evident. 
We conclude that $a\colon \downset a \rightarrow X$ is a linear map and its adjoint is 
$a^{*}\colon X\to \downset a$.

From Lemma \ref{lem:Sasaki projection facts}(a) we also obtain that the 
 map $a\colon \downset a\rightarrow X$ is a dagger monomorphism since 
 $a^{*}(a(u)) = \pi_a(u)=u.$
\end{proof}

\begin{lemma} \label{KernelLemma}
Let $f \colon X \to Y$ be a morphism of  orthomodular lattices. Then 
$\kernel f=\downset f^{*}(1)^{\perp}$ is an orthomodular lattice.
\end{lemma}
\begin{proof} Let $x \in X$. We compute: 
\begin{align*}
x\in \kernel f&\text{ if and only if } f(x)=0  \text{ if and only if } 
f(x)\perp 1\\ 
&\text{ if and only if } x\perp f^{*}(1) \text{ if and only if } 
x\leq f^{*}(1)^{\perp}.
\end{align*}
\end{proof}

 \begin{corollary}\label{SSasself}
     Let $X$ be an orthomodular lattice, and let 
     $f \colon X \to X$ be a self-adjoint morphism of  orthomodular lattices. Then 
$\kernel f=\downset f(1)^{\perp}$ and 
\mbox{$f(f(y^\perp)^\perp)\leq y$} for all $y\in X$.
 \end{corollary}
 \begin{proof} It is enough to check that 
 \mbox{$f(f(y^\perp)^\perp)\leq y$}. We compute:  \begin{align*} 
 f(f(y^\perp)^\perp)\leq y&%
  \text{ if and only if }  
  f(f(y^\perp)^\perp)\perp y^\perp \\
   &\text{ if and only if }   
   f(y^\perp)^\perp\perp f(y^\perp).
 \end{align*}
 \end{proof}

We show that \Cat{OMLatLin} has a {\it zero object} $\nul$; this means that there is, for any 
orthomodular lattice $X$, a unique morphism $\nul \to X$ and hence 
also a unique morphism $X \to \nul$. The zero object $\nul$ will be 
the one-element orthomodular lattice $\{0\}$.

Let us show that $\nul$ is indeed an 
initial object in $\Cat{OMLatLin}$. Let $X$ be an arbitrary orthomodular
lattice. By Corollary \ref{cor:lattice-adjoint}, the only linear map $f \colon \nul \to X$ satisfies $f(0)=0$. 
Since we may identify $\nul$ with $\downset 0$ we have 
by Lemma \ref{DownsetLemma} that $f$ is is a dagger monomorphism and 
it has an adjoint $f^* \colon X\to \nul$ defined by $f^*(x)=\pi_0(x)=0$.

\section{Foulis m-semilattices} \label{Foulis}

An {\it m-semilattice\/} is a join-semilattice $(S,\sqcup,0)$  with an associative
binary multiplication satisfying
$$
x\cdot\bigsqcup\limits_{i\in I}
x_i=\bigsqcup\limits_{i\in I}x\cdot
x_i\ \ \hbox{and}\ \ (\bigsqcup\limits_{i\in I}x_i)\cdot
x=\bigsqcup\limits_{i\in I}x_i\cdot x
$$
for all $x,\,x_i\in S,\,i\in I$ ($I$ is a finite set). 
We denote the respective order on $S$ by $\sqsubseteq$. 
The distributivity of $\cdot$ over finite joins implies that 
$\cdot$ is monotone with respect
to $\sqsubseteq$ and that $0$ is a zero in the semigroup 
$(S, \cdot)$.

An m-semilattice $S$ is called {\it unital\/} if
there is an element $e\in S$ such that
$$
e\cdot a = a = a\cdot e
$$
\noindent
for all $a\in S$.

By an {\it involutive  m-semilattice} will be meant
an m-semilattice $S$ together with a semigroup
 involution $^{*}$, i.e., $ s^{**}= s$ 
 and $(s \cdot t)^{*}= t^{*} \cdot s^{*}$ 
for all $s, t \in S$, satisfying
$$
(\bigsqcup\limits_{i\in I} x_{i})^{*}=%
\bigsqcup\limits_{i\in I} x_{i}^{*}
$$
\noindent for all $x_{i}\in S, \,i\in I$ ($I$ is a finite set).  
In particular, it follows that $0^{*}= 0$.
In the event that $S$
is also unital, then necessarily $e$ is selfadjoint, i.e.,
$$
e=e^{*}.
$$
Namely, $e^{*}\cdot x=(e^{*}\cdot x)^{**}=%
(x^{*}\cdot e^{**})^{*}=(x^{*}\cdot e)^{*}=(x^{*})^{*}=x$ and similarly $x\cdot e^{*}=x$ for all $x\in S$. 

We also define $s\leq t$ if 
and only if $s=t\cdot s$, and $s\perp t$ if 
and only if $0=s^{*}\cdot t$ for all $s, t\in S$. 
Evidently, $\leq$ is transitive and $\perp$ is symmetric. 
Namely, $s\leq t$ and $t\leq u$ implies $s=t\cdot s$ and 
$t=u\cdot t$. We conclude $s=t\cdot s=u\cdot t\cdot s=u\cdot s$. 
Observe also that, if $s\perp t$,
then $0 = 0^{*}= (s^{*}\cdot t)^{*}= t^{*}\cdot s^{**}%
= t^{*}\cdot s$, i.e., $t\perp s$. 

\begin{proposition}\label{inv}
     Let $X$ and $Y$ be orthomodular lattices. Then 
     \begin{enumerate}[label={\rm({\roman*})}]
     \item $\Cat{Lin}(X,Y)$ is a join-semilattice 
     (with respect to the point\-wise order $(Y, \sqsubseteq)^{X}$),
         \item $\Cat{Lin}(X)$ is an involutive unital 
         m-semilattice (with respect to the pointwise order, composition and taking adjoints).
     \end{enumerate}
 \end{proposition}
 \begin{proof} Ad (i): Clearly, the zero map $0_{X,Y}$ has as adjoint 
 the zero map $0_{Y,X}$. Hence $\Cat{Lin}(X,Y)$ has a smallest element. 
 Now, let $f, g\in \Cat{Lin}(X,Y)\subseteq Y^{X}$. Assume that $x\in X$ and $y\in Y$. 
 We compute with the help of Lemma \ref{lem:lattice-adjoint}: 
 \begin{align*}
   (f\vee g)(x) \perp y &\text{ if and only if }  f(x)\leq y^{\perp}\ 
   \text{and}\ g(x)\leq y^{\perp}\\ 
   &\text{ if and only if }   
   f^{*}(y) \leq x^{\perp} \text{ and }g^{*}(y) \leq x^{\perp}\\  
   &\text{ if and only if } %
   (f^{*}\vee g^{*})(y) \leq x^{\perp}.
\end{align*}
Therefore $f\vee g\in \Cat{Lin}(X,Y)$ and  $(f\vee g)^{*}=f^{*}\vee g^{*}$. 

Ad (ii): Evidently, $\Cat{Lin}(X)$ has by (i) arbitrary finite joins and is a monoid with respect to composition of linear maps. Moreover, 
${}^{*}$ is a semigroup involution on $\Cat{Lin}(X)$ by 
Theorem~\ref{OMLisdagger}.

Now, let 
$f,\,g_i\in \Cat{Lin}(X),\,i\in I$ ($I$ is a finite set). 
Assume that $x\in X$. We compute  with the help of Corollary  \ref{cor:lattice-adjoint}:
\begin{align*}
(f\circ \bigvee\limits_{i\in I} g_i)(x)&=%
f(\bigvee\limits_{i\in I}g_i(x))=\bigvee\limits_{i\in I}f(g_i(x))\\
&=%
\left(\bigvee\limits_{i\in I}f \circ g_i\right)(x)
\end{align*}
The remaining distributive law follows by the same arguments as before.
\end{proof}

An involutive monoid structure, augmented by the operation $\sai{-}$, was initially termed a \auv{Baer *-semigroup} by Foulis in his seminal works of the 1960s. This structure is now widely recognized as a \auv{Foulis semigroup}.

\begin{definition}
    \label{FoulisDef}
A \emph{Foulis semigroup} consists of a monoid $(S,\cdot, e)$ together with two endomaps $(-)^{*} \colon S\rightarrow S$
and $\sai{-} \colon S\rightarrow S$ satisfying for all $s,t\in S$:
\begin{enumerate}[label={\rm({\arabic*})}]
\item $e^{*} = e$ and 
$(s\cdot t)^{*} = t^{*}\cdot s^{*}$
  and $s^{**} = s$, making $S$ an involutive monoid;

\item $\sai{s}$ is a self-adjoint idempotent, i.e., 
  $\sai{s} \cdot \sai{s} = \sai{s} = \sai{s}^{*}$;

\item $0 \;\smash{\stackrel{\textrm{def}}{=}}\; \sai{e}$ is a zero
element: $0 \cdot s = 0 = s\cdot 0$;

\item $s\cdot t = 0$ iff $\ex{y}{t = \sai{s}\cdot y}$.
\end{enumerate}
For an arbitrary $t\in S$ put $t^{\perp}
  \,\smash{\stackrel{\textrm{def}}{=}}\, \sai{t^{*}} \in \sai{S}$.
  Hence from~(2) we get the equations $t^{\perp} \cdot t^{\perp} =
  t^{\perp} = (t^{\perp})^{*}$. From (2) and (3) we also infer 
  $0^{*} = \sai{e}^{*} =\sai{e} = 0$.
\end{definition}

\begin{definition}
				By a {\em  Foulis  m-semilattice} will be meant
				an involutive unital m-semilattice $S$ that with its multiplication $\cdot$ and its involution $^{*}$ forms a Foulis semigroup. 
                We will call elements of $\sai{S}=\set{\sai{t}}{t\in S} $ {\em Sasaki projections}.
					\end{definition}

\begin{remark}\label{remFouldef}\rm 
    We immediately see that an m-semilattice $S$ is a Foulis m-semilattice if and only if 
   it is a unital involutive m-semilattice $S$ 
together with an endomap ${-}^{\perp} \colon S\rightarrow S$ satisfying  for all $s,t\in S$:
\begin{enumerate}[label=({\arabic*})]
\item ${s}^{\perp}$ is a self-adjoint idempotent, {i.e.},~satisfies
  ${s}^{\perp} \cdot {s}^{\perp} = {s}^{\perp} = \big({s}^{\perp}\big)^{*}$;

\item $0\, {=} \, {e}^{\perp}$, i.e., ${e}^{\perp}$ is a zero with respect to $(S, \cdot)$;

\item $s\perp t$ if and only if  $ s^{*}\cdot t = 0$ if and only if $\ex{y}{t = {s}^{\perp}\cdot y}$.

\end{enumerate}
Since ${s}^{\perp} \cdot {s}^{\perp} = {s}^{\perp}$ we conclude with the help of (3) 
that $s\perp {s}^{\perp}$, i.e., $s^{*}\cdot {s}^{\perp}=0$ 
and $s^{\perp}\cdot {s}=0$. 
\end{remark}

A classical result in orthomodular lattice theory, first established over four decades ago in [13] and further elaborated in [4, Chapter II, Section 19], states that the endomorphisms $\Cat{Lin}(X)$ of any orthomodular lattice $X$ form a Foulis semigroup. The original proof utilized the theory of Galois connections.

It is well-established that for any orthomodular lattice $X$, the set of endomorphisms 
$\Cat{Lin}(X)$ forms a Foulis semigroup. This result has been known for over 60 years, as detailed in \cite{Foul} and \cite[Chapter~5, \S\S18]{OML}, where the construction is described in terms of Galois connections.

\begin{theorem}{\rm{\cite{Foul}, \cite{OML}}}\label{thm:linfoul}
    Let $X$ be an orthomodular lattice and let us define the endomap $\sai{-} \colon \Cat{Lin}(X)\rightarrow \Cat{Lin}(X)$ by 
 $\sai{s} =\pi_{s^*(1)^{\perp}}$ for all linear maps $s\in \Cat{Lin}(X)$. Then $(\Cat{Lin}(X),\circ, \idmap)$ is a Foulis semigroup  with respect to taking adjoints ${}^{*}$ and $\sai{-}$.
 \end{theorem}

 \begin{corollary}
     Let $X$ be an orthomodular lattice. Then 
     $\Cat{Lin}(X)$ is a Foulis m-semilattice.
 \end{corollary}
 \begin{proof}
     Use Proposition~\ref{inv}~(ii), Remark~\ref{remFouldef} and Theorem~\ref{thm:linfoul}.
 \end{proof}

 \begin{theorem}
\label{FoulisOMKerLem}
Let $S$ be a Foulis  m-semilattice. Then, 
for all $t, r\in S$ and $k\in \sai{S}$, 
\medskip

\noindent$\begin{array}{@{}r@{\,}c@{\,}c@{\,}c@{\,}c@{\,}c@{\,}c}
r\perp t&\Longleftrightarrow &r^{*}\cdot t=0
& \Longleftrightarrow &
t=[r^{*}]\cdot t=r^{\perp}\cdot t&\Longleftrightarrow &t\leq r^{\perp}%
\end{array}
{(*)}$
$$\begin{array}{@{}r@{\,}c@{\,}c@{\,}c@{\,}c@{\,}c@{\,}l}
t\leq  r&\Longleftrightarrow &t=r\cdot t
& \Longrightarrow &r^{\perp}= t^{\perp}\cdot r^{\perp} &
\Longleftrightarrow &
r^{\perp}\leq t^{\perp} \\
&\text{and}&
k^{\perp\perp}=k,
\end{array}\eqno{\phantom{*}(**)}$$
$$\begin{array}{rcl}
t\leq  r^{\perp}
& \Longleftrightarrow &
r\leq t^{\perp}. 
\end{array}\eqno{(***)}$$

\noindent{}and the subset $\sai{S}$ is an orthomodular lattice with the following structure.
$$\begin{array}{lrcl}
\mbox{Order} & k_{1}\leq k_{2} & \Leftrightarrow & k_{1} = k_{2}\cdot k_{1} \\
\mbox{Top} & e & = &  
   \sai{0} \\
\mbox{Orthocomplement\quad\quad} & k^{\perp} & = & \sai{k} \\
\mbox{Finite binary meet} & k_{1} \conjun k_{2} & = &  
   \big(k_{1} \cdot \sai{\sai{k_{2}}\cdot k_{1}}\big)^{\perp\perp}\\
\mbox{Finite join} & \bigvee X & = &  \sai{\sai{\bigsqcup X}}.
\end{array}$$
\end{theorem}
\begin{proof} The majority of this content is widely recognized in  \cite[Lemma 4.6.]{Jac} 
and its proof for Foulis semigroups. It remains only to show that 
$\sai{\sai{\bigsqcup X}}$ is the join of the finite subset 
$X\subseteq [S]$.
We put 
$t =   \bigsqcup X$. Since ${}^*$ is a $\sqcup$-endomorphism 
and the elements of $X$ are self-adjoint,, $t$ is self-adjoint. 
Therefore, $t^{\perp}= \sai{t}$, and since 
$\sai{t}^*= \sai{t}$, also $t^{\perp\perp}= \sai{t}^{\perp}= \sai{\sai{t}}$. 
From Remark \ref{remFouldef} we obtain that $t\cdot  \sai{t}=t\cdot t^{\perp}=0$.
\begin{enumerate}[label={\rm ({\roman*})}]
    \item    Let $r\in X$. 
    Since the multiplication $\cdot$ is order-preserving we have $r\cdot \sai{\bigsqcup X}%
    \sqsubseteq\bigsqcup X\cdot \sai{\bigsqcup X}=0$. Hence   $r\cdot \sai{\bigsqcup X}=0$. 
    Since $0^*= 0$ and both  $r$ and $\sai{t}$ are self-adjoint we conclude 
    that also $\sai{\bigsqcup X}\cdot r=0$. Using (*) we obtain 
    that $\sai{\sai{\bigsqcup X}}\cdot r=r$, i.e., $r\leq \sai{\sai{\bigsqcup X}}$. Hence $\sai{\sai{\bigsqcup X}}$ is an upper 
    bound of $X$.
     \item Let $u$ be an upper bound of $X$ in $\sai{S}$, i.e., $s=u\cdot s$ for all 
     $s\in X$.
    We compute: 
    \begin{align*}
     u\cdot t&= u\cdot {\bigsqcup X}= \bigsqcup \{u\cdot s\mid {s\in X}\}=%
    \bigsqcup \{s\mid {s\in X}\}\\
    &=t.
    \end{align*}
    From the conclusion that $t\leq u$, we derive from (**) that
    $\sai{\sai{\bigsqcup X}}=t^{\perp\perp}\leq u^{\perp\perp}=u$.
    Hence $\sai{\sai{\bigsqcup X}}=(\bigsqcup X)^{\perp\perp}$ is the smallest upper bound of $X$  in $\sai{S}$.
    \end{enumerate} 
\end{proof}

\section{m-semilattice modules}\label{semmod}

\begin{definition}\label{keymodule} Given a unital m-semilattice $(S, \sqcup, 0)$, a {\em  left $S$-module} 
	is a join-semilattice $(A, \vee, 0)$  and a 
	module action $\bullet\colon S\times A\longrightarrow A$ satisfying:
	\begin{itemize}
		\item[](A1) $s \bullet (\bigvee B)=\bigvee_{x\in B}(s \bullet x)$ 
		for every finite $B\subseteq A$ and  $s \in S$.
		\item[](A2) $(\bigsqcup T)\bullet a=\bigvee_{t\in T}(t\bullet a)$ 
		for every finite $T\subseteq S$ and  $a \in A$.
		\item[](A3) $u\bullet(v\bullet a)=(u\cdot v)\bullet a$ for every $u,v \in S$ and every $a\in A$.
		\item[](A4) $e \bullet a=a$ for all $a\in A$ (unitality).\\
	\end{itemize}	
\end{definition}
Right $S$-modules are defined similarly. Evidently, 
every  join-semilattice $A$   is a right  and left ${\mathbf 2}$-module. 
Here, ${\mathbf 2}$ is a 2-element chain, its multiplication 
is its meet and  involution is the identity map on it.

The following statement asserts that an orthomodular lattice exhibits a dual nature. It can be acted upon from the left by its own linear transformations and simultaneously from the right by a specific two-element structure. This two actions provides two distinct, yet compatible, means of transforming or modifying the elements within the lattice.

\begin{proposition}\label{rqinv}
	Let $X$ be a orthomodular lattice. Then 
	$X$ is a left $\Cat{Lin}(X)$-module and also a right $\Cat{2}$-module.
\end{proposition}
\begin{proof} By Proposition~\ref{inv}~(ii), $\Cat{Lin}(X)$ is a unital m-semi\-lattice.We define the action 
$\bullet\colon \Cat{Lin}(X)\times X\longrightarrow X$ by 
$f\bullet x=f(x)$ for all $f\in \Cat{Lin}(X)$ and all $x\in X$. 
The verification of conditions (A1)-(A4) is transparent.
\end{proof}

The following theorem reveals that the orthomodular lattice of all Sasaki projections within 
Foulis m-semilattice possesses a dual nature:
\begin{itemize}
\item Internal Action: Elements of Foulis m-semilattice itself can interact with and 'transform' the Sasaki projections.
\item External Action: The set of Sasaki projections can also be influenced by a simple two-element system.
\end{itemize}
This dual module structure provides valuable insights into the algebraic and structural properties of Foulis quantales and their associated Sasaki projections.

\begin{theorem} \label{thmFoulismod}
Let $S$ be a Foulis m-semilattice. Then $\sai{S}$ is a left $S$-module   with action $\bullet$ defined as 
$u\bullet k=(u\cdot k)^{\perp\perp}$ for all $u\in S$ and $k\in \sai{S}$  and also a right $\Cat{2}$-module.
\end{theorem}
\begin{proof}
{\bfseries Claim 1.} $e \bullet k=k$ for all $k\in \sai{S}$. 

\noindent{}Proof of Claim 1.  Let $k\in \sai{S}$. We compute, using (**), that: 
\begin{align*}
\phantom{xxxx}e \bullet k=(e\cdot k)^{\perp\perp}=k^{\perp\perp}=k.\hfill \phantom{xxxxxxxxxxxxxxxxi}\endclaim
\end{align*}

\medskip
\noindent{}{\bfseries Claim 2.} For all $r, s, t\in S$, 
$$\begin{array}{rclcl}
s\leq  t
& \Longrightarrow &
r\cdot s \leq (r\cdot s)^{\perp\perp}\leq (r\cdot t)^{\perp\perp}.
\end{array}\eqno{(\circled{{\small M1}})}$$
Proof of Claim 2.  From Remark \ref{remFouldef} we obtain $(r\cdot t)^{\perp}\cdot (r\cdot t)=0$. Hence also 
$(r\cdot t)^{\perp}\cdot (r\cdot t \cdot s)=0$. Since $t\cdot s=s$ we conclude that 
$(r\cdot t)^{\perp}\cdot (r\cdot s)=0$. From (*) 
and Remark~\ref{remFouldef}~(1) we conclude that $r\cdot s=(r\cdot t)^{\perp\perp}\cdot (r\cdot s)$, i.e., 
$r\cdot s \leq (r\cdot t)^{\perp\perp}$. In the special case $s = t$, we obtain 
$r \cdot s \leq (r · s)^{\perp\perp}$. Now applying ${}^{\perp\perp}$ to 
$r\cdot s \leq (r\cdot t)^{\perp\perp}$ 
and using (**) we have that 
$r\cdot s \leq (r\cdot s)^{\perp\perp}\leq 
(r\cdot t)^{\perp\perp\perp\perp}
= (r\cdot t)^{\perp\perp}$
since $(r\cdot t)^{\perp\perp}\in \sai{S}$. 

\hfill \endclaim

As a special case of Claim 2 with $s = t = e$, we obtain
$r \leq r^{\perp\perp}$ for all $r \in S$. From $0^{*}\cdot r=0$ we get that $r=0^{\perp} r$ for all $r\in S$, 
i.e., $0^{\perp}=e$. Moreover, 
we conclude that $r=0$ if and only if $r^{\perp\perp}=0$. 
Namely, if $r^{\perp\perp}=0$ then 
$r=r^{\perp\perp}\cdot r=0\cdot r=0$. Conversely, if 
$r=0$ then $r^{\perp\perp}=e^{\perp}=0$.

\medskip

\noindent{\bfseries Claim 3.} For all $s, t\in S$, 
$$\begin{array}{c}
(s\cdot  t)^{\perp\perp}=(s\cdot  t^{\perp\perp})^{\perp\perp}
\end{array}\eqno{(\circled{\small M2})}$$

\noindent{}Proof of Claim 3.  From Remark \ref{remFouldef} we have $(s\cdot  t)\perp (s\cdot  t)^{\perp}$. 
This means, all of the following equivalent statements hold:
\begin{align*}
&(s\cdot  t)^{*}\cdot (s\cdot  t)^{\perp}=0  %
\;\smash{\stackrel{\phantom{\text{Remark \ref{remFouldef}}}}{ \Longleftrightarrow}} t^{*}\cdot s^{*} \cdot (s\cdot  t)^{\perp}=0 \;\smash{\stackrel{{\text{(*)}}}{ \Longleftrightarrow}}\\
&t^{\perp}\cdot s^{*} \cdot (s\cdot  t)^{\perp}= s^{*} \cdot (s\cdot  t)^{\perp}
\;\smash{\stackrel{\text{(**)}}{ \Longleftrightarrow}}\; \\
&t^{\perp\perp\perp}\cdot s^{*} \cdot (s\cdot  t)^{\perp}= s^{*} \cdot (s\cdot  t)^{\perp}
\smash{\stackrel{\text{(*)}}{ \Longleftrightarrow}}
t^{\perp\perp}\cdot s^{*} \cdot (s\cdot  t)^{\perp}=0\\
 &\smash{\stackrel{\phantom{\text{\phantom{c}}}}{ \Longleftrightarrow}}\;  (s\cdot  t)^{\perp} \cdot s \cdot t^{\perp\perp}=0  \;\smash{\stackrel{{(*) }}{ \Longleftrightarrow}}\; 
 (s\cdot  t)^{\perp\perp} \cdot s \cdot t^{\perp\perp}=s \cdot t^{\perp\perp}\\
  &\smash{\stackrel{\text{(*)}}{ \Longleftrightarrow}}\;  s\cdot  t^{\perp\perp} \leq (s\cdot  t)^{\perp\perp}%
  \;\smash{\stackrel{\text{(**)}}{ \Longleftrightarrow}}\; %
  (s\cdot  t^{\perp\perp})^{\perp\perp} \leq (s\cdot  t)^{\perp\perp\perp\perp}\\ %
  &\smash{\stackrel{{\text{(**)}}}{ \Longleftrightarrow}}\; (s\cdot  t^{\perp\perp})^{\perp\perp} \leq (s\cdot  t)^{\perp\perp}.
\end{align*}

Since $t\leq t^{\perp\perp}$, (\circled{{\small M1}}) implies 
$(s\cdot  t)^{\perp\perp}\leq (s\cdot  t^{\perp\perp})^{\perp\perp}$, i.e., 
(\circled{\small M2}) holds. \hfill \endclaim

\medskip
\noindent{\bfseries Claim 4.} For all $s,t \in S$ and  $k\in \sai{S}$, 
\begin{align*}
    s\bullet (t\bullet k)=(s\cdot t)\bullet k.
\end{align*}

\noindent{}Proof of Claim 4.  We compute with the help of (M2): 
\begin{align*}
    s\bullet (t\bullet k)&=\big(s\cdot (t\bullet k)\big)^{\perp\perp}%
    =\big(s\cdot (t\cdot k)^{\perp\perp}\big)^{\perp\perp}\\%
    &=\big(s\cdot (t\cdot k)\big)^{\perp\perp}%
    =\big((s\cdot t)\cdot k\big)^{\perp\perp}=(s\cdot t)\bullet k. \hfill \phantom{x}\endclaim
\end{align*}

  \medskip

\noindent{}{\bfseries Claim 5.} For all  $T\subseteq {S}$, $T$ finite 
\begin{align*}
    \left(\bigsqcup T\right)^{\perp}= \bigwedge_{s\in T} s^{\perp}&\text{\quad and\quad } %
    \left(\bigsqcup T\right)^{\perp\perp}= \bigvee_{s\in T} s^{\perp\perp}. \tag{\circled{{\small M3}}}
\end{align*}

\noindent{}Proof of Claim 5.  Let $u\in \sai{S}$. Then 
by (***) $u\leq  \left(\bigsqcup T\right)^{\perp}$ if and only if 
$\bigsqcup_{s\in T} s \leq u^{\perp}$, which holds if and only if 
$\bigsqcup_{s\in T} s = u^{\perp}\cdot \bigsqcup_{s\in T} s= \bigsqcup_{s\in T} \left(u^{\perp}\cdot s\right)$. 

Remark \ref{remFouldef} implies that $\bigsqcup_{s\in T} u^{*}\cdot s =  \bigsqcup_{s\in T} \left(u^{*}\cdot u^{\perp}\cdot s\right)=0$. Hence 
we conclude that $u^{*}\cdot s=0$. Therefore by (*) $u \perp s$, thus $s \perp u$, i.e., 
by (*),  $u\leq s^{\perp}$ for all $s\in T$. 

Conversely, let $u\leq s^{\perp}$ for all $s\in T$. Then, by (*) $s \perp u$, thus $u \perp s$, i.e, by (*), 
$u^{*}\cdot s = 0$. Therefore, 
$ u^{*}\cdot\bigsqcup_{s\in T} s = \bigsqcup_{s\in T} (u^{*}\cdot s)=0$. We obtain that $u\perp \bigsqcup T$, i.e., 
using (*) and symmetry of $\perp$, we conclude 
$u\leq \left(\bigsqcup T\right)^{\perp}$. 

Altogether, $u\leq  \left(\bigsqcup T\right)^{\perp}$ if and only if  $u\leq  \bigwedge_{s\in T} s^{\perp}$. 

The second equation follows immediately from the first one and the fact that $\sai{S}$ is an
orthomodular lattice. \hfill \endclaim

\medskip
\noindent{}{\bfseries Claim 6.} For all $s \in S$ and $T\subseteq \sai{S}$, $T$ finite 
\begin{align*}
    s\bullet \left(\bigvee T\right)= \bigvee_{t\in T} \left(s\bullet t\right).
\end{align*}

\noindent{}Proof of Claim 6.  By utilizing (\circled{{\small M2}}) and (\circled{{\small M3}}), 
we are able to perform the following calculation:
\begin{align*}
    s&\bullet \left(\bigvee T\right)=s\bullet \left( \left(\bigsqcup T\right)^{\perp\perp}\right)=
    \left(s\cdot \left(\bigsqcup T\right)^{\perp\perp}\right)^{\perp\perp}\\%
    \smash{\stackrel{{\text{(M2)}}}{=}}& \left(s\cdot \left(\bigsqcup T\right)^{}\right)^{\perp\perp}=%
    \left(\bigsqcup_{t\in T}\left(s\cdot t\right)\right)^{\perp\perp}%
    \smash{\stackrel{{\text{(M3)}}}{=}}\bigvee_{t\in T} \left(s\cdot  t\right)^{\perp\perp}\\
     =&\bigvee_{t\in T} \left(s\bullet t\right).\hfill \phantom{xxxxxxxxxxxxxxxxxxxxxxxxxxxxxxxx}\endclaim
\end{align*}

\medskip
\noindent{}{\bfseries Claim 7.} For all $s \in \sai{S}$ and $T\subseteq {S}$, $T$ finite 
\begin{align*}
    \left(\bigsqcup T\right) \bullet s= \bigvee_{t\in T} \left(t\bullet s\right).
\end{align*}

\noindent{}Proof of Claim 7. Through the application of (\circled{{\small M2}}) and (\circled{{\small M3}}), we can carry out the following computation:
\begin{align*}
    &\left(\bigsqcup T\right) \bullet s=\big( \left(\bigsqcup T\right) \cdot s\big)^{\perp\perp}%
    =\left(\bigsqcup_{t\in T}\left(t\cdot s\right)\right)^{\perp\perp}\\
     \smash{\stackrel{{\text{(M3)}}}{=}}&\bigvee_{t\in T}\left(t\cdot s\right)^{\perp\perp}%
     = \bigvee_{t\in T} \left(t\bullet s\right).\hfill \phantom{xxxxxxxxxxxxxxxxxxxxi}\endclaim
\end{align*}

From Claims 1, 4, 6 and 7 we can conclude that $\sai{S}$ is a left $S$-module.
 \end{proof}

\begin{definition}\label{Sasaki action}
    Let $S$ be a Foulis m-semilattice and $u\in S$. 
    Then the map $\sigma_u:\sai{S}\to \sai{S}$, $y\mapsto u\bullet  y$ is called the \emph{Sasaki action} to $u\in S$.
\end{definition}

The following proposition describes key properties of the Sasaki action 
$\sigma_u$ when applied to a Sasaki projection $u$. 

\begin{proposition}\label{prop}
    Let $S$ be a Foulis m-semilattice and $u\in \sai{S}$.  
    Then the {Sasaki action} $\sigma_u$ is self-adjoint linear, idempotent 
    and $\image \sigma_u=\downset u$ in $\sai{S}$.
\end{proposition}
\begin{proof} Let $k,l \in \sai{S}$. We compute:
\begin{align*}
    \sigma_u(k)\perp l\;&\smash{\stackrel{\phantom{\text{Remark \ref{remFouldef}}}}{ \Longleftrightarrow}}\; \sigma_u(k)\leq l^{\perp}\;  %
    \;\smash{\stackrel{\text{by (**)}}{ \Longleftrightarrow}}\; \; 
    u\cdot k\leq l^{\perp} \\%
    &\;\;\smash{\stackrel{\text{def.}\; \leq}{ \Longleftrightarrow}}\; \; 
    l^{\perp}\cdot u\cdot k=u\cdot k %
    \;\smash{\stackrel{\text{by (*)}}{ \Longleftrightarrow}}\; \; 
    l\cdot u\cdot k=0\\
    &\smash{\stackrel{\text{by involution}}{ \Longleftrightarrow}}\; \; 
    k\cdot u\cdot l=0%
    \;\smash{\stackrel{\text{by (*)}}{ \Longleftrightarrow}}\; 
    k^{\perp}\cdot u\cdot l=u\cdot l\\
    &\;\;\smash{\stackrel{\text{def.}\; \leq}{ \Longleftrightarrow}}\; \; %
    u\cdot l\leq k^{\perp} %
    \;\smash{\stackrel{\text{by (**)}}{ \Longleftrightarrow}}\; 
    \; \sigma_u(l)\leq k^{\perp}\\
    &\smash{\stackrel{\phantom{\text{Remark \ref{remFouldef}}}}%
    { \Longleftrightarrow}}\; \; \sigma_u(l)\perp k.
\end{align*}
Hence $\sigma_u$ is self-adjoint linear. We also have by (A3):
\begin{align*}
    \sigma_u(\sigma_u(k))=u\bullet (u\bullet k))=(u\cdot u)\bullet k=%
    u\bullet k=\sigma_u(k).
\end{align*}
We conclude that $\sigma_u$ is idempotent.

Since $k\leq e$ in $\sai{S}$ we obtain from Claim 2 in Theorem \ref{thmFoulismod} 
that $\sigma_{u}(k)=(u\cdot k)^{\perp\perp}\leq (u\cdot e)^{\perp\perp}=%
u^{\perp\perp}=u$, i.e., $\image \sigma_u\subseteq \downset u$. 
Conversely, let $k\in \downset u$. Then $u\cdot k=k$. Hence also 
$\sigma_{u}(k)=(u\cdot k)^{\perp\perp}=k^{\perp\perp}=k$, 
i.e., $\downset u\subseteq\image \sigma_u$. 
\end{proof}

\section{Conclusions}\label{Conclusions}

Our research demonstrates that, analogous to the findings in \cite{Jac}, $\Cat{OMLatLin}$ is a dagger category. We associated with every orthomodular lattice $X$ a Foulis m-semilattice $\Cat{Lin}(X)$ of endomorphisms of $X$, which serves as a quantale. $X$ then becomes a left $\Cat{Lin}(X)$-module. Consequently, we introduced a natural fuzzy-theoretic aspect to the theory of orthomodular lattices.

For future research, we intend to explore several problems, including:
\begin{enumerate}
    \item Following \cite{Jac}, we plan to show that 
    \begin{enumerate}
        \item the category $\Cat{OMLatLin}$ is a dagger kernel category, and
        \item for a dagger kernel category \Cat{D}, the kernel subobject functor $\KSub(-)$ is a functor $\cat{D} \rightarrow \Cat{OMLatLin}$.
    \end{enumerate}
    
    \item Our analysis revealed that $\Cat{OMLatLin}$ behaves 
    as a quantaloid as introduced in \cite{rosenthal2}. We aim to show 
    that $\Cat{OMLatLin}$ is an involutive semi-quantaloid.
\end{enumerate}

\section*{Acknowledgment}
The first author acknowledges support from the Czech Science Foundation (GAČR) project 23-09731L ``Representations of algebraic semantics for substructural logics''. The research of the second author
was funded in part by the Austrian Science Fund (FWF) 10.55776/PIN5424624 and the Czech Science Foundation
(GACR) 25-20013L. The third author acknowledges support from the Masaryk University project MUNI/A/1457/2023.


\end{document}